\documentclass[11pt]{amsart}
\usepackage[margin=1in]{geometry}
\usepackage{mathtools}
\usepackage{algorithm}
\usepackage{amssymb}
\usepackage{algpseudocode}
\usepackage{amsfonts}
\usepackage{hyperref}
\algrenewcommand\algorithmicrequire{\textbf{Input}:}
\algrenewcommand\algorithmicensure{\textbf{Output}:}

\usepackage{todonotes}
\usepackage[normalem]{ulem}

\DeclareMathOperator{\initial}{in}
\DeclareMathOperator{\gr}{gr}
\DeclareMathOperator{\leadterm}{lt}
\DeclareMathOperator{\leadmonomial}{lm}
\DeclareMathOperator{\conv}{conv}
\DeclareMathOperator{\exponent}{exp}
\DeclareMathOperator{\vol}{vol}
\DeclareMathOperator{\torexp}{torexp}
\DeclareMathOperator{\tor}{torsion}
\DeclareMathOperator{\torfree}{torfree}

\newcommand{\bR}{{\mathbb R}}

\newcommand{\bQ}{{\mathbb Q}}
\newcommand{\bN}{{\mathbb N}}
\newcommand{\bZ}{{\mathbb Z}}

\theoremstyle{definition}
\newtheorem{theorem}{Theorem}[section]
\newtheorem{definition}[theorem]{Definition}
\newtheorem{remark}[theorem]{Remark}
\newtheorem{lemma}[theorem]{Lemma}
\newtheorem{proposition}[theorem]{Proposition}
\newtheorem{corollary}[theorem]{Corollary}
\newtheorem{example}[theorem]{Example}
\newtheorem*{standinghypothesis}{Standing Hypotheses}

\usepackage{multicol}
\usepackage{cleveref}

\title{Subalgebra and Khovanskii bases equivalence}

\author[Alstad]{Colin Alstad}
\address{220 Parkway Drive, Clemson University, Clemson, SC 29634}
\email{calstad@clemson.edu}

\author[Burr]{Michael Burr}
\address{220 Parkway Drive, Clemson University, Clemson, SC 29634}
\email{burr2@clemson.edu}

\author[Clarke]{Oliver Clarke}
\address{The University of Edinburgh, James Clerk Maxwell Building, Edinburgh EH9 3FD}
\email{oliver.clarke@ed.ac.uk}

\author[Duff]{Timothy Duff}
\address{Department of Mathematics, University of Washington, Seattle, WA 98195}
\email{timduff@uw.edu}

\begin{document}
\maketitle

\begin{abstract}
The main results of this paper establish a partial correspondence between two previously-studied analogues of Gr\"{o}bner bases in the setting of algebras: namely, subalgebra (aka SAGBI) bases for quotients of polynomial rings and Khovanskii bases for valued algebras.
We aim to bridge the gap between the concrete, computational aspects of the former and the more abstract theory of the latter.
Our philosophy is that most interesting examples of Khovanskii bases can also be realized as subalgebra bases and vice-versa.
We also discuss the computation of Newton-Okounkov bodies, illustrating how interpreting Khovanskii bases as subalgebra bases makes them more amenable to the existing tools of computer algebra.
\end{abstract}

\section{Introduction}
\emph{Subalgebra bases} (sometimes also called \emph{canonical bases} or \emph{SAGBI bases}) were originally introduced as analogues to Gr\"obner bases for polynomial algebras independently by Kapur and Madlener~\cite{KapurMadlener:1989} and Robbiano and Sweedler~\cite{RobbianoSweedler:1990}.  This concept was further generalized to quotient polynomial rings by Stillman and Tsai~\cite{StillmanTsai:1999} and to \emph{Khovanskii bases} of valued algebras by Kaveh and Manon~\cite{KavehManon:2019}. 

There are several existing implementations of subalgebra bases for polynomial algebras in computer algebra systems: two implementations \cite{sagbi.lib,BrunsConca:2024} using \textsc{Singular} \cite{Singular}, a forthcoming implementation \cite{BigattiRobbiano:2022} in \textsc{CoCoA} \cite{CoCoA}, and an implementation \cite{SubalgebraBasesSource,Burretal:2023} in \textsc{Macaulay2} \cite{Macaulay2} by several of the authors of this paper. 
Among these implementations, we note that the recent work in ~\cite{BrunsConca:2024} reports impressive runtimes compared to the alternatives on a test suite of challenging examples.
On the other hand, the package~\cite{SubalgebraBasesSource}, described in~\cite{Burretal:2023}, also handles subalgebra bases for quotient rings.
This level of generality is needed for the computations in the present paper.

Showcasing the generality of Khovanskii bases,~\cite[Example 7.7]{KavehManon:2019} constructs finite Khovanskii bases for the standard invariant ring of the alternating group $A=k[x,y,z]^{A_3}$. Viewed as a subalgebra $A \subseteq k[x,y,z]$, a finite subalgebra basis for $A$ does not exist~\cite{Gobel:1995}.
However, in ~\Cref{ex:alternating}, we show that there is more to the story: if we present $A$ as the \emph{quotient} of a polynomial ring, the Khovanskii bases in question are also subalgebra bases in the sense of~\cite{StillmanTsai:1999}.

Our main goal is to establish explicit connections between the two previously-mentioned notions of Khovanskii bases and subalgebra bases for quotient rings, with an eye towards leveraging existing implementations. We show in Theorem \ref{thm:khovanskiirealization} that the most common cases of Khovanskii bases, namely those arising from valuations that satisfy our \emph{standing hypotheses} in~\Cref{sec:KhovanskiitoSubalgebra}, can also be realized as subalgebra bases of quotients of polynomial rings.
We note that the same hypotheses are satisfied by the valuations constructed using tropical geometry in~\cite[Theorem 1]{KavehManon:2019}.

As a partial converse,~\Cref{cor:khovanskii-from-sagbi} gives sufficient conditions under which subalgebra bases for quotient rings are also Khovanskii bases.
Note that this result holds unconditionally in the original setting of polynomial rings.
\Cref{sec:eequivalence-valuations} contains further discussion of the relationship between the monomial orders and valuations appearing in our constructions.
Finally, in~\Cref{sec:NewtonOkounkov}, we apply our results to computing Newton-Okounkov bodies.

\section{Background}
Fix a field $k$ and let $R\vcentcolon=k[x_1,\dots,x_n]$ be a polynomial ring with a monomial order $<$.  We fix the convention that monomial orders use the maximum convention, that is, $\leadterm(f+g)\leq\max\{\leadterm(f),\leadterm(g)\}$, provided none of $f$, $g$, and $f+g$ are zero.

Let $A\vcentcolon=k[f_i\colon i\in \mathcal I]$ be a subalgebra of $R$, where $\mathcal I$ is some index set.  The \emph{initial algebra} of $A$ is the algebra generated by the leading terms of $A$, that is,
\begin{equation}\label{eq:initialAlgebra}
\leadterm(A)\vcentcolon=k[\leadterm(f) \colon f\in A].
\end{equation}
Here, we use the notation $\leadterm(A)$ instead of the more common notation $\initial_<(A)$, as we reserve the latter notation for valuations and Khovanskii bases, cf. \cite{Sturmfels:1996}.

\begin{definition}
A set $\{g_j\}_{j\in\mathcal{J}}\subseteq A$ is a {\em subalgebra basis} for $A$ with respect to $<$ if the leading terms $\{\leadterm(g_j)\}_{j\in\mathcal{J}}$ generate the initial algebra of $A$, that is,
\begin{equation}
\leadterm(A)=k[\leadterm(g_j) \colon j\in\mathcal{J}].\label{eq:SAGBIdef}
\end{equation}
\end{definition}

One of the most significant differences between subalgebra bases and Gr\"obner bases is that subalgebra bases do not receive the benefits of the Noetherian property.  Therefore, a finitely generated algebra may not have a finite subalgebra basis under any term ordering, see, for example, \cite[Example 1.20]{RobbianoSweedler:1990}.

Many of the standard algorithms from Gr\"obner basis theory have analogues in subalgebra basis theory.  For instance, polynomial long division is replaced by subduction.  The subduction algorithm provides a rewriting of a polynomial $f\in R$ with respect to a subalgebra basis $\{g_j\}_{j\in\mathcal{J}}$.  The result of the algorithm is a finite sum
$$
f=\sum_{\alpha\in\bN^{\mathcal{J}}} c_\alpha g^{\alpha} + r,
$$
with the following properties: (1) $\alpha_j$ is zero for all but finitely many indices, (2) $c_\alpha$ is zero for all but finitely many values of $\alpha$, (3) if $r\not=0$, then $\leadterm(r)\leq\leadterm(f)$, (4) the nonzero terms of $r$ are not in the initial algebra $\leadterm(A)$, (5) for every $\alpha$ with $c_\alpha\not=0$, $\leadterm(f)\geq \leadterm(g^\alpha)$, and (6) the values of $\leadmonomial(g^\alpha)$ with $c_\alpha\not=0$ are distinct.  The finiteness in these definitions comes about since monomials orders are well-orders, even when the subalgebra basis itself has infinitely many polynomials.
Subduction has many standard properties of other polynomial rewriting procedures: for example, the remainder $r$ is zero if and only if $f\in A$.  In fact, the remainder $r$ is independent of the choice of $c_\alpha$.  Therefore, when $f\in A$, the remainder is zero and exactly one $\alpha$ has both $c_\alpha\not=0$ and $\leadmonomial(f)=\leadmonomial(g^\alpha)$.

The term order induces a filtration on $A$ by $k$-vector spaces as follows: For any $\alpha\in\bN^n$,
\begin{equation}\label{eq:filtration}
F_{\leq\alpha}(A)\vcentcolon=\{f\in A:\leadterm(f)\leq x^\alpha\}\cup\{0\}.
\end{equation}
$F_{<\alpha}(A)$ is defined similarly, by replacing $\leq$ by $<$ in Equation (\ref{eq:filtration}).  Then the {\em associated graded algebra} is
\begin{equation}\label{eq:associatedgraded}
\gr_<(A)\vcentcolon=\bigoplus_{\alpha\in\bN^n}F_{\leq\alpha}(A)/F_{<\alpha}(A).
\end{equation}

\subsection{Subalgebra bases for quotient rings}
In \cite{StillmanTsai:1999}, Stillman and Tsai extended the definition of subalgebra bases to quotients of polynomial rings.  They present their theory for quotient rings over commutative Noetherian domains, but we restrict our discussion to quotient rings over fields.

Suppose that $k$, $R$, and $<$ are defined as above, and let $I$ be an ideal of $R$.  Let $A$ be a subalgebra of the quotient $R/I$.  Stillman and Tsai then consider the following sequence of maps in order to define the notion of a leading term:
$$
\begin{tabular}{ccccccc}
$R/I$&
$\xrightarrow{\sim}$&
$R$&
$\xrightarrow{\leadterm}$&
$R$&
$\xrightarrow{q}$&
$R/\leadterm(I)$\\
$[f]$&$\mapsto$&
$\tilde{f}$&$\mapsto$&
$\leadterm(\tilde{f})$&$\mapsto$&
$q(\leadterm(\tilde{f}))$.
\end{tabular}
$$
The map $\sim$ indicates the normal form of $f$ in terms of standard monomials (and not an isomorphism).  The function $\leadterm$ selects the largest nonzero term under $<$ of its input polynomial.  Finally, the map $q$ is the quotient map.  The image of $[f]$ under this sequence of maps is defined to be its {\em leading term}, denoted by $\leadterm([f])$.

With the notion of a leading term, the definition of the {\em initial algebra} of $A$ is identical to the definition from polynomial algebras, as in Equation~\eqref{eq:initialAlgebra}.
\begin{definition}
A set $\{[g_j]\}_{j\in\mathcal{J}}\subseteq A$ is a {\em subalgebra basis} for $A$ with respect to $<$ if the leading terms $\{\leadterm([g_j])\}_{j\in\mathcal{J}}$ generate the initial algebra of $A$, as defined as in Equation~\eqref{eq:SAGBIdef}.
\end{definition}
We highlight that $\leadterm([f])$ is an element of $R/\leadterm(I)$ while $\leadterm(f)$ is a monomial in $R$.  Therefore, in the quotient case, both of the algebras appearing in Equations~\eqref{eq:initialAlgebra} and \eqref{eq:SAGBIdef} are subalgebras of $R/\leadterm(I).$

Subduction and the associated graded have similar definitions and properties to the polynomial case.  For example, for $[f]\in R/I$ and $\{[g_j]\}_{j\in\mathcal{J}}$ a subalgebra basis, the algorithm produces
$$
\tilde{f}=\sum_{\alpha\in\bN^{\mathcal{J}}}c_\alpha \tilde{g}^\alpha + r + h,
$$
with the additional property that $h\in I$ and $r$ is composed of standard monomials for $I$.  The definition of leading terms for quotients can be used to define a filtration as in Equation \eqref{eq:filtration} and associated graded as in Equation \eqref{eq:associatedgraded}.  In this case, the filtration is defined as
\begin{equation}\label{eq:filtrationquotient}
F_{\leq\alpha}(A)\vcentcolon=\{[f]\in R/I:\leadterm(\tilde{f})\leq x^\alpha\}.
\end{equation}
$F_<(A)$ is defined similarly, by replacing $\leq$ by $<$ in Equation (\ref{eq:filtrationquotient}).  Stillman and Tsai show that the associated graded $\gr_<(A)$ is isomorphic to $\leadterm(A)$, see the algebraic remark in \cite[Section 2]{StillmanTsai:1999}.

\subsection{Khovanskii bases}
In \cite{KavehManon:2019}, Kaveh and Manon adapted the ideas of subalgebra bases to finitely generated valued $k$-algebras as follows: Let $A$ be a finitely generated $k$-algebra and domain.  Suppose, in addition, that $A$ is equipped with a valuation $\nu\vcentcolon A\setminus\{0\}\rightarrow\bQ^r$ that lifts the trivial valuation on $k^\times$.  Moreover, we assume that $\bQ^r$ is given a total ordering $\succ$ so that the image of $\nu$ is maximum-well-ordered.  We fix the convention that $\nu$ is a min-valuation, that is,  $\nu(f+g)\succeq\min\{\nu(f),\nu(g)\}$, provided none of $f$, $g$, and $f+g$ are zero.

This valuation induces a filtration where, for any $a\in\bQ^r$,
\begin{equation}\label{eq:filtrationKhovanskii}
F_{\succeq a}(A)\vcentcolon=\{f\in A:\nu(f)\succeq a\}\cup\{0\}.
\end{equation}
$F_{\succ}(A)$ is defined similarly, by replacing $\succeq$ by $\succ$ in Equation (\ref{eq:filtrationKhovanskii}).  The {\em associated graded algebra} is
$$
\gr_\nu(A)=\bigoplus_{a\in\bQ^r} F_{\succeq a}(A)/F_{\succ a}(A).
$$
We often include the assumption that $\nu$ has \emph{one-dimensional leaves}, meaning each of the summands above is a vector space of dimension at most 1. 
This assumption is frequently satisfied, see e.g.~\cite[Theorem 2.3]{KavehManon:2019}, 
and it has useful algorithmic consequences.

\begin{definition}
A set of nonzero elements $\{g_j\}_{j\in\mathcal{J}}$ is a {\em Khovanskii basis} for $A$ with respect to $\nu$ if the set of images 
$g_j+F_{\succ \nu(g_j)}$ for $j \in \mathcal{J}$
generate $\gr_\nu(A)$ as an algebra.
\end{definition}

As in the subalgebra basis case, the subduction algorithm can be adapted to Khovanskii bases. Any element of the algebra $f\in A$ can be rewritten as a polynomial in the Khovanskii basis $\{g_j\}_{j\in\mathcal{J}}$,
$$
f=\sum_{\alpha\in\bN^{\mathcal{J}}}c_\alpha g^\alpha,
$$
with the following properties: (1) $\alpha_j$ is zero for all but finitely many indices, (2) $c_\alpha$ is zero for all but finitely many values of $\alpha$, (3) for every $\alpha$ with $c_\alpha\not=0$, $\nu(f)\preceq \nu(g^\alpha)$, and (4) the values of $\nu(g^\alpha)$ with $c_\alpha\not=0$ are distinct.
 The finiteness comes from the fact that the image of $\nu$ is well-ordered, and the distinctness comes from the one-dimensional leaves assumption. Unlike subalgebra bases, the algebra $A$ might not be a subalgebra of some larger algebra, so the subduction rewriting can only be performed on elements of $A$.

\section{Khovanskii bases as subalgebra bases}\label{sec:KhovanskiitoSubalgebra}

Let $A$ be a finitely generated $k$-algebra and domain, as in the definition of Khovanskii bases.  Suppose that $\{g_1,\dots,g_m\}$ is a finite Khovanskii basis for $A$. 
Our goal is to find a subalgebra basis in the sense of \cite{StillmanTsai:1999} that reflects the structure and properties of $A$.

We maintain the following \emph{standing hypotheses} on a valuation in order to simplify theorem statements.
\begin{standinghypothesis}\label{standinghypothesis}
Let $A$ be a $k$-algebra equipped with a valuation $\nu:A\setminus\{0\}\rightarrow\bQ^r$, where $\bQ^r$ is given a total order $\prec$.
We say the valued algebra $\nu:A\setminus\{0\}\rightarrow\bQ^r$ satisfies the \emph{standing hypotheses} if (1) the image $\nu (A \setminus\{0\}) \subseteq\bQ^r$ is maximum-well-ordered, (2) $A$ is a finitely-generated domain, (3) $\nu $ lifts the trivial valuation on $k^\times$, and (4) $\nu$ has one-dimensional leaves.    
\end{standinghypothesis}

Let $R\vcentcolon=k[x_1,\dots,x_m]$ be a {\em presentation ring} for $A$ corresponding to our Khovanskii basis, meaning that there is a \emph{presentation map} $\pi:R\rightarrow A$ defined by $\pi(x_i)=g_i$.  Let $I$ be the kernel of this map.  Using the valuation $\nu,$ we define a monomial order on  $R$.
\begin{definition}
Let $R$ be a presentation ring for a finite Khovanskii basis with respect to $\nu $ satisfying the standing hypotheses.
We define a monomial order $<$ on $R$ {\em induced by the valuation $\nu$} as follows: $x^\alpha>x^\beta$ if $\nu(\pi(x^\alpha))\prec\nu(\pi(x^\beta))$ or $\nu(\pi(x^\alpha))=\nu(\pi(x^\beta))$ and $x^\alpha>'x^\beta$ for some fixed, tie-breaking monomial order $<'$ on $R$.
\end{definition}
We verify that $<$ defines a monomial order.
It is a total ordering since both the image of $\nu$ is totally ordered and $<'$ induces a total ordering on $\bN^m$.  
Also, $<$ refines the partial order given by divisibility due to the property $\nu(f g) = \nu (f) + \nu (g).$
Finally, we note that $\nu$ lifts the trivial valuation on $k$ and its image is well-ordered.
Thus, there cannot be any $f\in A \setminus k$ such that $\nu(f)\succ 0$. 
Since $<'$ satisfies $1\leq' x^\alpha$ for all $\alpha$, we conclude that $1\leq x^\alpha$ for any $\alpha\in\bN^m$.

\begin{remark}
One of the main results in \cite{KavehManon:2019} is the construction of the valuation $\nu_M$ corresponding to a prime cone of the tropicalization of $I$, where $I$ is the kernel defined above.  This a push-forward of a quasi-valuation~\cite[Definition 2.26]{KavehManon:2019} $\tilde{\nu}_M$ defined on the presentation ring $R$.  When the image of $\nu_M$ is maximum-well-ordered, the monomial order $<$ defined here is a refinement of the order on monomials given by $-\tilde{\nu}_M$.
\end{remark}

We now consider $R/I$ as a subalgebra of itself. 
For a Khovanskii basis $\{ g_1, \ldots , g_m \},$ we
observe that $\{[x_1],\dots,[x_m]\}$ form a subalgebra basis for $R/I$.  To verify this, we note that if $\leadterm([x_i])\not=x_i+\leadterm(I)$, then $x_i$ is not a standard monomial.  Therefore, for all $f\in R$, $x_i$ does not appear in $\leadterm(\tilde{f})$, and so $\leadterm([f])$ can be represented without $x_i$.  Therefore, $\{[x_1],\dots,[x_m]\}$ and $\{[x_1],\dots,\widehat{[x_i]},\dots[x_m]\}$ generate the same subalgebra of $\gr_>(R/I)$.  Continuing recursively, we find that for every $f$ in $R$, $\leadterm([f])$ can be represented by a monomial in the variables $x_i$ such that $\leadterm([x_i])=x_i+\leadterm(I)$.  Hence, these variables generate $\gr_>(R/I)$ and form a subalgebra basis.

Therefore, not only are $R/I$ and $A$ isomorphic as $k$-algebras, but the presentation map $\pi$ also takes the subalgebra basis $\{[x_1],\dots,[x_m]\}$ for $R/I$ to the Khovanskii basis $\{g_1,\dots,g_m\}$ for $A$.

\begin{remark}\label{rem:notKhovanskii}
The argument above does not use the assumption that $\{g_1,\dots,g_m\}$ form a Khovanskii basis.  Therefore, the set of $\{[x_i]\}$ where $\leadterm([x_i])=x_i+\leadterm(I)$ always form a subalgebra basis for $R/I$.  
\end{remark}

Next, we show that the grading and filtration given by this monomial order match that of the valuation.

\begin{lemma}\label{lem:leadtermequality}
With the standing hypotheses and notation as above, let $[f]\in R/I$ be nonzero, then $\nu(\pi(f))=\nu(\pi(\leadterm(\tilde{f})))$.
\end{lemma}
\begin{proof}
Using subduction, we write $\pi(f)=\sum_{\alpha\in\bN^m}c_\alpha g^\alpha$ as a finite sum.  By construction, there is exactly one $\alpha$ with $c_\alpha\not=0$ and $\nu(g^\alpha)=\nu(\pi(f))$.  For all other $\alpha$, either $c_\alpha=0$ or $\nu(g^\alpha)\succ \nu(\pi(f))$.  We define $\overline{f}=\sum_{\alpha\in\bN^m}c_\alpha x^\alpha$.  By the properties of subduction and the fact that the monomial order uses the valuation to order, we observe that $\nu(\pi(\leadterm(\overline{f})))=\nu(\pi(f))$.

We first show that  $\nu(\pi(\leadterm(\tilde{f})))\preceq \nu(\pi(f))$ by contradiction.  If the other inequality held, then every term of $\tilde{f}$ would map to an element of $A$ with valuation greater than $\nu(\pi(f))$ since the monomial order is based on valuation-order.  But then, since $\nu$ is a min-valuation, $\nu(\pi(\tilde{f}))\succ \nu(\pi(f))$, which is not possible.

On the other hand, if $\nu(\pi(\leadterm(\tilde{f})))\prec \nu(\pi(f))$, then $\leadterm(\tilde{f}-\overline{f})=\leadterm(\tilde{f})$ since the monomial order $>$ is based on the valuation order.  However, $\tilde{f}-\overline{f}\in I$, which contradicts the fact that $\tilde{f}$ is a linear combination of standard monomials.  
\end{proof}

In the previous proof, we note that generally $\tilde{f}\ne \overline{f}$. 
We can only conclude that the images of their leading terms under $\pi $ must have the same valuation.  This relationship can be seen more precisely in the following lemma, which shows that the filtrations of $R/I$ and $A$ are compatible with the quotient $\pi$.

\begin{lemma}
With the standing hypotheses and notation as above,
$$F_{\leq\alpha}(R/I)=\left\{[f]\in R/I:\pi(f)\in F_{\succeq\nu(\pi(x^\alpha))}(A)\right\}.$$
\end{lemma}
\begin{proof}
By definition, $[f]\in F_{\leq\alpha}$ if and only if $\leadterm(\tilde{f})\leq x^\alpha$.  By the construction of the monomial order, this is true if and only if either $\nu(\pi(\leadterm(\tilde{f})))\succ\nu(\pi(x^\alpha))$ or $\nu(\pi(\leadterm(\tilde{f})))=\nu(\pi(x^\alpha))$ and $\leadterm(\tilde{f})\leq' x^\alpha$.  By Lemma \ref{lem:leadtermequality}, $\nu(\pi(f))=\nu(\pi(\leadterm(\tilde{f})))$, and, in either case, $\nu(\pi(f))\succeq \nu(\pi(x^\alpha))$, that is, $\pi(f)\in F_{\succeq\nu(\pi(x^\alpha))}(A).$

On the other hand, suppose that $\pi(f)\in F_{\succeq\nu(\pi(x^\alpha))}(A),$ but $[f]\not\in F_{\leq\alpha}$.  This implies that $\nu(\pi(f))\succeq \nu(\pi(x^\alpha))$, but $\leadterm(\tilde{f})>x^\alpha$.  By the construction of the monomial order and Lemma \ref{lem:leadtermequality}, if $\nu(\pi(f))\succ \nu(\pi(x^\alpha))$, then $\leadterm(\tilde{f})<x^\alpha$. It then follows that $\nu(\pi(f))= \nu(\pi(x^\alpha))$ and $\leadterm(\tilde{f})>'x^\alpha$.

Since $\nu$ has one-dimensional leaves, there is some $\lambda\in k$ so that either $\nu(\pi(f)-\lambda\pi(x^\alpha))\succ\nu(\pi(x^\alpha))$ or $\pi(f)-\lambda\pi(x^\alpha)=0$.  Using subduction, we write $\pi(f)-\lambda\pi(x^\alpha)=\sum_{\beta\in\bN^m}c_\beta g^\beta$ as a finite sum.  By the properties of subduction, for all $\beta$ with $c_\beta\not=0$, $\nu(\pi(f)-\lambda\pi(x^\alpha))\preceq \nu(g^\beta)$.  Let $h=\sum_{\beta\in\bN^m}c_\beta x^\beta$.  Then, by construction, $\tilde{f}-\lambda x^\alpha-h\in I$.

We now show that $\leadterm(\tilde{f}-\lambda x^\alpha-h)=\leadterm(\tilde{f})$.  In particular, since $\leadterm(\tilde{f})>x^\alpha$, $\lambda x^\alpha$ cannot be the leading term of this sum.  In addition, by the properties of subduction, for all $\beta$ with $c_\beta\not=0$, $\nu(\leadterm(\tilde{f}))\prec \nu(\pi(f)-\lambda\pi(x^\alpha))\preceq \nu(g^\beta)$.  This implies that $\leadterm(\tilde{f})>c_\beta x^\beta$ whenever $c_\beta\not=0$.  Therefore, no $c_\beta x^\beta$ can be the leading term of $\tilde{f}-\lambda x^\alpha-h$.

Since $\leadterm(\tilde{f}-\lambda x^\alpha-h)=\leadterm(\tilde{f})$ and $\tilde{f}-\lambda x^\alpha-h\in I$, $\leadmonomial(\tilde{f})$ cannot be a standard monomial of $I$.  This contradicts the definition of $\tilde{f}$, and so $[f]\in F_{\leq\alpha}$.
\end{proof}

By construction, the elements of $R/I$ are in bijective correspondence with those of $A$.  Hence, the elements of $F_{\leq\alpha}(R/I)$ are in bijective correspondence with the elements of $F_{\succeq\nu(\pi(x^\alpha))}(A)$.  By the previous result, we see that if $\nu(\pi(x^\alpha))=\nu(\pi(x^\beta))$, then $F_{\leq\alpha}(R/I)=F_{\leq\beta}(R/I)$ since the right-hand-side of the equality is the same for both $\alpha$ and $\beta$.

Suppose that $\leadmonomial(\tilde{f})=x^\alpha$ and $\nu(\pi(x^\alpha))=\nu(\pi(x^\beta))$.  Since $[f]\in F_{\leq\alpha}(R/I)=F_{\leq\beta}(R/I),$ it follows that $x^\alpha\leq x^\beta$.  In other words, $x^\alpha$ is the smallest monomial with respect to $<$ whose valuation is $\nu(\pi(x^\alpha))$.

On the other hand, suppose that $x^\beta$ is not the smallest monomial with respect to $<$ with valuation $\nu(\pi(x^\beta))$.  The observation above implies that if $\nu(\pi(f))=\nu(x^\beta)$, then $\leadterm(\tilde{f})<x^\beta$.  Therefore, $[f]\in F_{<\beta}(R/I)$.  Putting this together, we have the following conclusions:

\begin{corollary}\label{cor:notminimal}
Using the standing hypotheses and notation as above, suppose that $x^\alpha$ and $x^\beta$ are such that $\nu(\pi(x^\alpha))=\nu(\pi(x^\beta))$ and $x^\alpha<x^\beta$.  Then
$$
F_{\leq\alpha}(R/I)=F_{\leq\beta}(R/I)=F_{<\beta}(R/I).
$$
\end{corollary}

\begin{corollary}\label{cor:minimal}
Using the standing hypotheses and notation as above, suppose that $x^\alpha$ is the smallest monomial with respect to $<$ with valuation $\nu(\pi(x^\alpha))$.  Then,
$$
F_{<\alpha}(R/I)=\{[f]\in R/I:\pi(f)\in F_{\succ(\pi(x^\alpha))}(A)\}.
$$
\end{corollary}

In the case of Corollary \ref{cor:notminimal}, we see that $F_{\leq\beta}(R/I)/F_{<\beta}(R/I)$ is trivial.  On the other hand, in the case of Corollary \ref{cor:minimal}, we conclude that
$$
F_{\leq\alpha}(R/I)/F_{<\alpha}(R/I)\simeq F_{\succeq\nu(\pi(x^\alpha))}(A)/F_{\succ\nu(\pi(x^\alpha))}(A).
$$
These two observations directly imply that the associated graded algebras for both $R/I$ under $<$ and $A$ under $\succ$ are equal.  Collecting these results, we have the following:

\begin{theorem}\label{thm:khovanskiirealization}
Consider a valued algebra $\nu:A\setminus\{0\}\rightarrow\bQ^r$ satisfying the standing hypotheses, and suppose that $A$ has a finite Khovanskii basis with respect to $\nu.$ 
Then there is a polynomial ring $R$, ideal $I$ of $R$, and monomial order $<$ such that that
\[
\text{\emph{(1)} } R/I\simeq A\text{, and} \quad
\text{\emph{(2)} } \gr_{\nu}(A)\simeq\gr_{<}(R/I).
\]
\end{theorem}

\begin{remark}
We focus on the case where $A$ has a finite Khovanskii basis for computational reasons, but the theory can be extended to the case where $A$ has an infinite Khovanskii basis.  Then the polynomial ring $R$ would be a (countably) infinitely-generated polynomial ring.  This would also require extending the definition of subalgebra bases from \cite{StillmanTsai:1999}.  The definitions carry over {\em mutatis mutandis}, but we leave the details to the interested reader.
\end{remark}

\begin{example}\label{ex:alternating}
Consider the following example from~\cite[Example 7.7]{KavehManon:2019}: Let $A$ be the subalgebra of \(k[z_{1}, z_{2}, z_{3}]\) consisting of polynomials that are invariant under the action of $A_3$.  That is, \(A = k[e_{1}, e_{2}, e_{3}, y]\) where
\begin{align*}
  e_{1} &= z_{1} + z_{2} + z_{3}, &e_{2} &= z_{1}z_{2} + z_{1}z_{3} + z_{2}z_{3},\\
  e_{3} &= z_{1}z_{2}z_{3},&y &= (z_{1} - z_{2})(z_{1} - z_{3})(z_{2}-z_{3}).
\end{align*}
Let \(R = k[x_{1}, x_{2}, x_{3}, x_{4}]\) be the presentation ring of \(A\) where $\pi:R\rightarrow A$ with $\pi(x_i)=e_i$ and $\pi(x_4)=y$.
The kernel of the map $\pi$ is the principal ideal $I=\langle f \rangle ,$ where
\begin{equation}f=\underline{x_{1}^{2}x_{2}^{2}} - \underline{4x_{2}^{3}} - 4x_{3}x_{1}^{3} + 18x_{1}x_{2}x_{3} - 27x_{3}^{2} - x_{4}^{2}.\label{eq:discriminant}\end{equation}
The tropical variety \(\mathcal{T}(I) \subseteq \bR^4\) contains three maximal prime cones, and hence, by~\cite[Theorem 1]{KavehManon:2019}, the set \(\left\{ e_{1}, e_{2}, e_{3}, y \right\}\) is a Khovanskii basis for each of the valuations constructed from these cones.
Moreover, none of these valuations are induced by monomial order on $k[z_1,z_2,z_3]$ since a result of G\"{o}bel implies such subalgebra bases are always infinite~\cite{Gobel:1995}.
On the other hand, any valuation $\nu :A \to \bQ^2$ constructed from these prime cones corresponds to many different monomial orders $<$ on $R$ satisfying the conclusions of~\Cref{thm:khovanskiirealization}.
Consider, for instance, the prime cone generated by the rays \(\bR_{\ge 0}( -3, -6, 14, -9)\) and \(\bR_{\ge 0} (22, -2, -3, -3)\).
A suitable monomial order $<$ can be constructed from the weight matrix
$$M = \begin{pmatrix}
0 & 2 & 2 & 3 \\
    1 & 4 & 1 & 6
\end{pmatrix},$$
where $x^{\alpha } < x^{\beta }$ if $M \alpha $ is lexicographically smaller than $M \beta ,$ and any fixed monomial order is used to break ties.

We see that \(\left\{[x_{1}], [x_{2}], [x_{3}], [x_{4}]\right\}\) forms a subalgebra basis for $R/I$, which corresponds under $\pi$ to the Khovanskii basis $\{ e_1, e_2, e_3, y\}$ for $A$. Although the valuation on $A$ is not induced by a monomial order on $k[z_1,z_2,z_3]$, there is another ring $R$ and monomial order on this ring, which does induce the valuation on $A$.
\end{example}
Two remarks further illustrate the relationship between $\nu$ and $<$.
\begin{remark}\label{rem:altvaluation}
\Cref{lem:leadtermequality} implies the following characterization (cf.~\cite[Equation (3.2)]{KavehManon:2019}) of $\nu $ in terms of the presentation ring: 
$$\nu (g) = \max\{\nu(\pi(\leadterm(h))):h\in R\text{ and }\pi(h)=g\} \text{ for } g\in A \setminus \{0\}.$$
\end{remark}
\begin{remark}\label{rem:attainedTwice}
For every nonzero $f\in I$, the two largest monomials $x^\alpha$ and $x^\beta$ of $f$ with respect to $<$ must satisfy $\nu(\pi(x^\alpha))=\nu(\pi(x^\beta))$.
In~\Cref{ex:alternating}, these are the underlined terms of Equation~\eqref{eq:discriminant}.
\end{remark}

The construction of~\Cref{thm:khovanskiirealization} assumes it is known {\em a priori} that $\{g_1,\dots,g_m\}$ forms a Khovanskii basis. Without this assumption, the construction is merely existential. We next state a criterion that can identify when $\{g_1,\dots,g_m\}$ form a Khovanskii basis.

\begin{proposition}
For a valued algebra $\nu:A\setminus\{0\}\rightarrow\bQ^r$ satisfying the standing hypotheses, let $\{g_1, \ldots, g_m\}$ be a finite set of nonzero generators for $A$.
Let $R\vcentcolon=k[x_1,\dots,x_m]$ be the presentation ring for these generators, $I$ the kernel of the presentation map, and $<$ a monomial order induced by $\nu.$  The set $\{g_1,\dots,g_m\}$ forms a Khovanskii basis for $A$ if and only if $\gr_{\nu}(A)\simeq\gr_{<}(R/I)$.
\end{proposition}
\begin{proof}

By Theorem \ref{thm:khovanskiirealization}, if \(\left\{g_{1}, \ldots g_{m}\right\}\) is a Khovanskii basis of \((A, \nu)\), then  \(\gr_{\nu}(A) \cong \gr_{<}(R/I)\). On the other hand, by Remark \ref{rem:notKhovanskii} $\{[x_i]\}$ with $x_i$ a standard monomial form a subalgebra basis for $R/I$.  If the map $\leadterm(R/I) \rightarrow \gr_{\nu}(A)$ defined by $[x_i]+\leadterm(I) \mapsto g_i+F_{\succ\nu(g_i)}$ is an isomorphism, then $\{g_i+F_{\succ\nu(g_i)}\}$ generate $\gr_{\nu}(A)$. Hence $g_1,\dots, g_m$ is a Khovanskii basis.
%
\end{proof}

A Khovanskii or subalgebra basis is \emph{minimal} if none of its proper subsets form a Khovanskii or subalgebra basis for the same algebra.
Our previous constructions respect minimality.
\begin{proposition}
Let $\nu:A\setminus\{0\}\rightarrow\bQ^r$ be a valued algebra satisfying the standing hypotheses and $\{g_1,\dots,g_m\}$ be a finite Khovanskii basis for $A$.  Let $R\vcentcolon=k[x_1,\dots,x_m]$ be the presentation ring and$I$ the kernel of the presentation map.  There exists a monomial order $<$ induced from $\prec$ such that $\{g_1,\dots,g_m\}$ is a minimal Khovanskii basis if and only if $\{[x_1],\dots,[x_m]\}$ is a minimal subalgebra basis.  
\end{proposition}
\begin{proof}
Let $<$ be any monomial order induced by $\nu$.  Suppose there is an $i$ so that $\{[x_1],\dots,\widehat{[x_i]},\dots,[x_n]\}$ is a subalgebra basis for $R/I$.  This means that there is some $\alpha$ with $\alpha_i=0$ so that 
$\leadterm([x])^\alpha = \leadterm([x_i])$.  Since $\leadterm(I)$ is a monomial ideal, we conclude that $\leadterm(\widetilde{x}^\alpha)=\leadterm(\widetilde{x}_i)$.  We observe by Lemma \ref{lem:leadtermequality} that 
$$
\nu(\pi(x^\alpha))=\nu(\pi(\widetilde{x}^\alpha))=\nu(\pi(\leadterm(\widetilde{x}^\alpha)))=\nu(\pi(\leadterm(\widetilde{x_i})))=\nu(\pi(x_i)).
$$
Rewriting this statement in terms of Khovanskii bases gives that $\nu(g^\alpha)=\nu(g_i).$  Since $\nu$ has one-dimensional leaves, there is some $\lambda$ so that either $\nu(\lambda g^\alpha-g_i)\succ \nu(g_i)$ or $\lambda g^\alpha-g_i=0$.  In other words, $\lambda g^\alpha$ and $g_i$ have the same image in the associated graded algebra.  Since $\lambda g^\alpha$ does not involve $g_i$, we conclude that $\{g_1,\dots,\hat{g}_i\dots,g_m\}$ generates the same image in the associated graded, and these elements also form a Khovanskii basis.


On the other hand, suppose that $\{g_1,\dots,\hat{g}_i,\dots,g_m\}$ is a Khovanskii basis.  By applying subduction to $g_i$, we have $g_i=\sum_{\alpha\in\bN^m}c_\alpha g^\alpha$ as a finite sum such that for every $\alpha$ with $c_\alpha\not=0$, we have $\alpha_i=0$.  By the properties of subduction, for each $\alpha$ with $c_\alpha\not=0$, $\nu(g^\alpha)\succeq \nu(g_i)$.  Moreover, by the properties of the valuation, there is a unique $\beta$ with $c_\beta\not=0$ where equality is attained.  Now, consider the polynomial $f=x_i-\sum_{\alpha\in\bN^m}c_\alpha x^\alpha$.  By construction, $f\in I$ and $\leadterm(f)$ is either $c_\beta x^\beta$ or $x_i$.  Since the valuation of the images of these terms is the same, their order is determined by the fixed tie-breaking monomial order $<'$ on $R$.  We may choose the tie-breaking order to have $x_i>'c_\beta x^\beta$, for instance, using an elimination order. Since $x_i$ is the leading term of $f$, $x_i$ is not a standard monomial, and, by the argument preceding Remark \ref{rem:notKhovanskii}, $\leadterm([x_i])\not=x_i+\leadterm(I)$ and $[x_i]$ can be dropped from the subalgebra basis.  We iteratively apply this procedure, dropping one term of the Khovanskii basis and a corresponding subalgebra basis generator until both are minimal.
\end{proof}

\begin{remark}
The proof above shows that when $<$ is chosen appropriately, the Khovanskii basis and subalgebra basis elements are in bijective correspondence with each other.
\end{remark}

\section{Subalgebra bases as Khovanskii bases}\label{sec:subalgebra-to-khovanskii}
Let $R\vcentcolon=k[x_1,\dots,x_m]$.  Suppose that $R/I$ is a finitely generated $k$-algebra and domain with a monomial order $<$.  In this case, we can apply the theory of Khovanskii bases directly to $R/I$, provided we can find a suitable valuation on $R/I$.  

A motivating attempt would be to use $\tilde{\mu}:R/I\setminus\{[0]\}\rightarrow \bZ^m$ defined as $[f]\mapsto -\exponent(\leadmonomial(\tilde{f})),$ where $\exponent$ denotes the exponent of the input monomial.  In many cases, however, this is not a valuation.  In particular, suppose that $x^\alpha$ and $x^\beta$ are standard monomials, but their product $x^{\alpha+\beta}$ is not a standard monomial.  In this case,
$\tilde{\mu}([x^\alpha])+\tilde{\mu}([x^\beta])=-(\alpha+\beta),$ but this does not equal $\tilde{\mu}([x^{\alpha+\beta}])$ since $\leadmonomial(\widetilde{x^{\alpha+\beta}})\not=x^{\alpha+\beta}$.  We proceed to fix this deficiency.

\begin{definition}
  \label{def:toricExponents}
  Let $f\in R$ be nonzero and not a monomial.  Suppose that the two largest leading monomials of $f$ with respect to $<$ are $x^{\alpha_1}$ and $x^{\alpha_2}$, with $x^{\alpha_1}>x^{\alpha_2}$ (cf.~\Cref{rem:attainedTwice}).  We define the \emph{toric exponent} of $f$ to be $\torexp(f) = \alpha_1 - \alpha_2 \in \bZ^m$.
  %
\end{definition}
The key object in our construction is the following lattice:
$$
K\vcentcolon=\mathbb{Z}\{\torexp(f):f\in I\}\subseteq\mathbb{Z}^m.
$$
We then define the \emph{torsion-free} portion of $\mathbb{Z}^m/K$ as $$\torfree(\mathbb{Z}^m/K)\vcentcolon=(\mathbb{Z}^m/K)/\tor(\mathbb{Z}^m/K).$$
From this, we define the map 
$
\mu:R/I\setminus\{[0]\}\rightarrow \torfree(\mathbb{Z}^m/K)
$
where $\mu([f])$ maps to the image of $\tilde{\mu}([f])$ in this quotient.  We now define an order on the image of $\mu$.  In particular, suppose that $\mathfrak{a}$ and $\mathfrak{b}$ are in the image of $\mu$.  We say that $\mathfrak{a}\prec\mathfrak{b}$ if the smallest monomial $x^\alpha$ with $\mu([x^\alpha])=\mathfrak{a}$ is greater than the smallest monomial $x^\beta$ with $\mu([x^\beta])=\mathfrak{b}$.

We observe that the monomial $x^\alpha$, as defined above, is a standard monomial.  In particular, for any $[f]\in\mu^{-1}(\mathfrak{a})$, it follows that $[\leadterm(\tilde{f})]\in\mu^{-1}(\mathfrak{a})$.  We note that $\leadterm(\tilde{f})$ is both a standard monomial and smaller than $\leadterm(f)$. 

\begin{theorem}\label{thm:whenSAGBIisKhovanskii}
  Let \(R := k[x_{1}, \ldots, x_{m}]\) with monomial order \(<\).  Suppose $I$ is a prime, monomial-free ideal of $R$. Let $K$ be defined as above. Define $\mu:R/I\rightarrow \torfree(\mathbb{Z}^m/K)$ as above.  If for every nonzero $\mathfrak{a}\in\torfree(\mathbb{Z}^m/K)$ in the image of $\mu$ there is a unique standard monomial $x^\alpha$ such that $\mu([x^\alpha])=\mathfrak{a}$, then $\mu$ is a valuation on $R/I$.
\end{theorem}
\begin{proof}
Suppose that $[f_1],[f_2]\in R/I$, and let $x^{\alpha_i}=\leadmonomial(\widetilde{f_i}).$  Since our monomial orders use the maximum convention, $$
\leadmonomial(\widetilde{f_1+f_2})=\leadmonomial(\widetilde{f_1}+\widetilde{f_2})\leq\max\{x^{\alpha_1},x^{\alpha_2}\}.
$$
Rewriting this in terms of $\mu$, it follows that $$\mu([f_1+f_2])\succeq\min\{\mu([f_1]),\mu([f_2])\}.$$

Now, suppose that $\leadmonomial(\widetilde{f_1f_2})=x^\gamma$.  If $x^{\alpha+\beta}$ is a standard monomial, then $x^\gamma=x^{\alpha+\beta}$, and $\mu([f_1f_2])$ is the image of $\alpha+\beta$ in $\torfree(\mathbb{Z}^m/K)$, which is the sum of the images of $\alpha$ and $\beta$.

On the other hand, if $x^{\alpha + \beta}$ is not a standard monomial, then we consider 
\[h = x^{\alpha + \beta} - \widetilde{x^{\alpha + \beta}} = x^{\alpha + \beta} - x^{\gamma} - \text{lower order terms}\]
    Note that \(h\) is in \(I\), and it is neither zero nor a monomial.  Therefore, \(\torexp(h) = \alpha + \beta - \gamma\) is in \(K\).  From this, it follows that \(\mu([f_{1}]) + \mu([f_{2}]) = \mu([f_{1}][f_{2}])\).
\end{proof}

\begin{remark}
In the statement of~\Cref{thm:whenSAGBIisKhovanskii}, the assumption on the unique standard monomial in the preimages of $\mu$ is a version of the one-dimensional leaves assumption.
\end{remark}

\begin{corollary}\label{cor:khovanskii-from-sagbi}
Let \(R := k[x_{1}, \ldots, x_{m}]\) with monomial order \(<\).  Suppose $I$ is a prime, monomial-free ideal of $R$. Let $K$ be defined as above. Define $\mu:R/I\rightarrow \torfree(\mathbb{Z}^m/K)$ as above.  Suppose that for every nonzero $\mathfrak{a}\in\torfree(\mathbb{Z}^m/K)$ in the image of $\mu$ there is a unique standard monomial $x^\alpha$ such that $\mu([x^\alpha])=\mathfrak{a}$. Then $\{ [x_1], \ldots , [x_m]\}$ is a Khovanskii basis with respect to $\mu.$
\end{corollary}
\begin{proof}
Let $S\subseteq \{x_1, \ldots , x_m\}$ consisting of variables that are also standard monomials.  By~\Cref{rem:notKhovanskii} $\{[x_i]\}_{i\in S}$ is a subalgebra basis with respect to $<$.
Consider an element $\mu ([f])$ in the image of $\mu$, where $f\in R$. We have that $\mu ([f]) = \mu ([\tilde{f}]) = \mu ([\leadterm (\tilde{f})])$. 
We write $\leadterm (\tilde{f})$ as a product of variables in $S$ as follows:
$\leadterm (\tilde{f}) = \prod_{i\in S} x_i^{\alpha_i}$.
Hence $$[f] + F_{\succ \mu ([f])} = \left[\prod_{i\in S} x_i^{\alpha_i}\right] + F_{\succ \mu ([f])} = \prod_{i\in S} \left( [x_i] + F_{\succ \mu ([x_i])}\right)^{\alpha_i}.
$$
Therefore, $\{[x_i]\}_{i\in S}$ generate the associated graded.
\end{proof}

\section{Equivalence of valuations}\label{sec:eequivalence-valuations}
When the monomial order on $R/I$ is constructed as in Section \ref{sec:KhovanskiitoSubalgebra}, then $R/I\simeq A$ has two valuations on it: $\mu$ and $\nu$.  We show that these valuations are linearly equivalent.  As a first step, we simplify the construction of the lattice $K$ as above.

\begin{lemma}\label{lem:Kspan}
Consider a valued algebra $\nu:A\setminus\{0\}\rightarrow\mathbb{Q}^r$ satisfying the standing hypotheses, and suppose that $A$ has a finite Khovanskii basis $\{g_1,\dots,g_m\}$ with respect to $\nu$.  Let $R\vcentcolon=k[x_1,\dots,x_m]$ be the presentation ring for this basis.  Let $K$ be defined as above.  Then,
$$
K = \bZ\{\alpha - \beta \colon \alpha, \beta \in \mathbb{Z}^{m}, \, \,  \nu(\pi(x^{\alpha})) = \nu(\pi(x^{\beta}))\}.
$$
\end{lemma}
\begin{proof}
By Remark \ref{rem:attainedTwice}, every toric exponent of an element in $I$ is of the desired form.  On the other hand, suppose that $\nu(\pi(x^\alpha))=\nu(\pi(x^\beta))$ with $\alpha\not=\beta$.  Since $\nu$ has one-dimensional leaves, there exists a $\lambda$ such that $\nu(\pi(x^\alpha)-\lambda \pi(x^\beta))\succ \nu(\pi(x^\alpha))$ or $\pi(x^\alpha)-\lambda \pi(x^\beta)=0$.  We write
$
\pi(x^\alpha)-\lambda \pi(x^\beta)=\sum_{\gamma\in\bN^m}c_\gamma g^\gamma
$ as a finite sum using subduction.  By the properties of subduction, for all $\gamma$ with $c_\gamma\not=0$, $\nu(\pi(x^\alpha)-\lambda \pi(x^\beta))\succeq\nu(g^\gamma).$  Let $h=x^\alpha-\lambda x^\beta-\sum_{\gamma\in\bN^m}c_\beta x^\gamma$.  By construction, $h\in I$ and $\torexp(h)=\alpha-\beta$, so $\alpha-\beta\in K$.
\end{proof}

\begin{remark}\label{rem:polyexists}
The proof of this lemma shows that if $\nu(\pi(x^\alpha))=\nu(\pi(x^\beta))$, with $\alpha\not=\beta$, then there is a polynomial $f\in I$ whose two leading monomials are $x^\alpha$ and $x^\beta$.
\end{remark}

We further simplify the construction of $K$ in terms of a Gr\"obner basis for $I$. This simplification is particularly useful for the construction of the Newton-Okounkov body in Section \ref{sec:NewtonOkounkov}.

\begin{corollary}
Consider a valued algebra $\nu:A\setminus\{0\}\rightarrow\mathbb{Q}^r$ satisfying the standing hypotheses, and suppose that $A$ has a finite Khovanskii basis $\{g_1,\dots,g_m\}$ with respect to $\nu$.  Let $R\vcentcolon= k[x_1,\dots,x_m]$ be the presentation ring for this basis and $I$ the kernel of the presentation map $\pi:R\rightarrow A$.  Let $K$ be defined as above, and $\{f_1,\dots,f_l\}$ a Gr\"obner basis for $I$.  $K$ is generated by $\{\torexp(f_j)\}_{j=1}^l$ as a $\bZ$-lattice.
\end{corollary}
\begin{proof}
Let $a$ be in the image of $\nu$, and suppose that $\alpha_1,\dots,\alpha_n\in\bN^m$ are the exponents of all monomials $x^\alpha$ such that $\nu(\pi(x^\alpha))=a$ and ordered by $x^{\alpha_i}<x^{\alpha_{i+1}}$.  By induction on $l$, we show that all differences of the form $\alpha_{i_1}-\alpha_{i_2}$ with $i_1,i_2\leq l$ are in the $\bZ$-lattice generated by the toric exponents of the Gr\"obner basis.  The base case of $l=1$ is vacuously true.

We assume the claim is true for $l\geq 1$ and consider the case of $l+1$.  By Remark \ref{rem:polyexists}, there is some polynomial $h$ whose leading monomial is $x^{\alpha_{l+1}}$.  By the property of being a Gr\"obner basis, there is some $f_j$ such that $\leadmonomial(f_j)$ divides $x^{\alpha_{l+1}}$.  Therefore, there is some $x^\delta$ so that $x^{\alpha_{l+1}}=x^\delta\leadmonomial(f_j)$.  By Remark \ref{rem:attainedTwice}, the second largest monomial of $x^\delta f_j$ is $x^{\alpha_i}$ for some $i<l+1$.  Therefore, $\alpha_{l+1}-\alpha_i=\torexp(x^\delta f_j)=\torexp(f_j)$.  The inductive hypothesis then implies that the claim is true for the case of $l+1$.
\end{proof}

\begin{remark}\label{rem:quotientrepresentative}
Suppose that $f\in R$ so that $\nu(\pi(f))=\nu(\pi(\leadterm(f)))$.  Then, $\nu(\pi(\leadterm(f)))=\nu(\pi(\leadterm(\tilde{f})))$.  By Remark \ref{rem:polyexists}, there is a polynomial $f\in I$ whose leading monomials are $\leadmonomial(f)$ and $\leadmonomial(\tilde{f})$.  This implies that the image of $-\exp(\leadmonomial(f))$ in $\torfree(\bZ^m/K)$ equals $\mu([f])$.  In this case, it is not necessary to replace $f$ by $\tilde{f}$ in our computations.
\end{remark}

We now show that the two valuations defined on $R/I\simeq A$ are linearly equivalent.  This indicates that we may use subalgebra bases for quotient rings as a computational replacement for Khovanskii bases without losing information.

\begin{theorem}\label{thm:affinetransformation}
Consider a valued algebra $\nu:A\setminus\{0\}\rightarrow\mathbb{Q}^r$ satisfying the standing hypotheses, and suppose that $A$ has a finite Khovanskii basis $\{g_1,\dots,g_m\}$ with respect to $\nu$.  Let $R\vcentcolon= k[x_1,\dots,x_m]$ be the presentation ring for this basis and $I$ the kernel of the presentation map $\pi:R\rightarrow A$.  Let $K$ and $\mu$ be defined as above.  Then, $\mu$ and $\nu$ are linearly equivalent, that is, there is an invertible linear transformation $\phi$ from the span of the image of $\mu$ to the span of $\nu$ such that $\nu\circ\pi(f)=\phi\circ\mu([f])$ for all $f\in R$.
\end{theorem}
\begin{proof}
Suppose that the image of $\mu$ is of rank $r$ and that the variables $\{x_1,\dots,x_r\}$ are standard monomials such that $\bZ\{\mu([x_i])\}_{i=1}^r$ is of rank $r$.  We define the map $\phi$ as $\phi(\mu([x_i]))=\nu(\pi(x_i))$, and extend it by linearity.

We first show that $\{\nu(\pi(x_i))\}_{i=1}^r$ is independent by contradiction.  Suppose that there is a nontrivial sum $\sum_{i=1}^r c_i\nu(\pi(x_i))=\sum_{i=1}^r c_i\nu(g_i)=0$.  By scaling, we may assume that each $c_i$ is an integer.  We separate the positive and negative parts of the coefficients, where $\alpha_i=\max\{0,c_i\}$ is the vector of positive coefficients and $\beta_i=\alpha_i-c_i$ is the vector of negative coefficients.  It follows that $\nu(g^\alpha)=\nu(g^\beta)$.  Since $\pi(x^\alpha)=g^\alpha$ and $\pi(x^\beta)=g^\beta$, Remark \ref{rem:polyexists} implies that there is an $f$ in $I$ whose leading monomials are $x^\alpha$ and $x^\beta$.  Therefore, $\alpha-\beta=c\in K$.  By Remark \ref{rem:quotientrepresentative}, $\mu([x^\alpha])=\sum_i \alpha_i\mu([x_i])$, which is the image of $\alpha$ in $\torfree(\bZ^m/K)$, and $\mu([x^\beta])=\sum_i \beta_i\mu([x_i])$, which is the image of $\beta$ in $\torfree(\bZ^m/K)$.  Therefore, $\sum_i c_i\mu([x_i])$ is the image of $\alpha-\beta$ in $\torfree(\bZ^m/K)$.  This image is trivial, which contradicts the assumption that $\{\mu([x_i])\}_{i=1}^r$ is independent.

Suppose that there is a nontrivial $\sum_{i=1}^m c_i\mu([x_i])=0$.  By scaling, we may assume that each $c_i$ is an integer.  We separate the positive and negative parts of the coefficients where $\alpha_i=\max\{0,c_i\}$ is the vector of positive coefficients and $\beta_i=\alpha_i-c_i$ is the vector of negative coefficients.  By Remark \ref{rem:quotientrepresentative}, we note that  $\sum_i \alpha_i\mu([x_i])$ is the image of $\alpha$ in $\torfree(\bZ^m/K)$ and $\sum_i\beta_i\mu([x_i])$ is the image of $\beta$ in $\torfree(\bZ^m/K)$.  Since $\sum_i \alpha_i\mu([x_i])=\sum_i\beta_i\mu([x_i])$, we have that the image of $\alpha-\beta$ is zero in $\torfree(\bZ^m/K)$.  In other words, there is some nonzero integer $s\in\bZ$ so that $s(\alpha-\beta)\in K$.  From Lemma \ref{lem:Kspan}, it follows that $\nu(\pi(x^{s\alpha}))=\nu(\pi(x^{s\beta}))$.  Since $\nu(\pi(x^{s\alpha}))=s\nu(\pi(x^{\alpha}))$ and $\nu(\pi(x^{s\beta}))=s\nu(\pi(x^{\beta}))$, we conclude that $\nu(\pi(x^{\alpha}))=\nu(\pi(x^\beta)).$  Thus, by Remark \ref{rem:polyexists}, $\alpha-\beta\in K$ and $\sum_i\alpha_i\nu(g_i)=\sum_i\beta_i\nu(g_i)$ or that $\sum_i c_i\nu(g_i)=0.$

Therefore, since any $\mu([x_i])$ can be written in terms of the basis $\{\mu([x_1]),\dots,\mu([x_r])\}$, $\nu(g_i)$ is the corresponding linear combination of $\{\nu(g_1),\dots,\nu(g_r)\}$.  Hence, $\{\nu(g_1),\dots,\nu(g_r)\}$ also form a full rank sublattice of the lattice generated by $\{\nu(g_1),\dots,\nu(g_m)\}$.  These relationships additionally imply that the transformation $\phi$ takes $\mu([x_i])$ to $\nu(g_i)$ for all $1\leq i\leq m$.
\end{proof}

\section{Newton-Okounkov bodies}\label{sec:NewtonOkounkov}
One of the most important invariants of a graded, valued algebra $A$ is its Newton-Okounkov body \cite{KavehManon:2019,KavehKhovanskii:2012,Anderson:2013,BurrKhovanskiiHomotopies:2023}. 
This is a convex body which captures homological and geometric data of $A$. For instance, the normalized volume of the Newton-Okounkov body is the asymptotic growth rate of the Hilbert function for the algebra, see, for instance \cite[Theorem 2.23]{KavehManon:2019} and \cite[Theorem 4.9]{KavehKhovanskii:2012}.  We show how to compute the Newton-Okounkov body of a graded algebra using the constructions of the previous sections.

We follow the construction of the Newton-Okounkov body as in \cite{KavehManon:2019}. 
Consider a valued algebra $\nu ' : A ' \setminus \{ 0 \} \to \bZ^r$ satisfying the standing hypotheses, and a positively graded algebra $A=\bigoplus_{i\geq 0}A_i$ where $A_i \subseteq A'$ for all $i.$  
We extend $\nu '$ to valuation, which also satisfies the standing hypotheses, $\nu:A \setminus \{ 0 \} \to \bN \times \bZ^r$.
We decompose $f\in A \setminus \{ 0 \}$ into homogeneous components, $f=\sum_{i=0}^m f_i$ with $f_i\in A_i$ and $f_m\not=0$, and define $\nu(f)\vcentcolon=(m,\nu'(f_m))$. We order $\bN \times\bZ^r$ so that $(m,a)\succ (n,b)$ if $m<n$ or $m=n$ and $a\succ b$.
\begin{definition}
The {\em Newton-Okounkov body} associated to $A$ and $\nu '$ is the closed, convex set
\begin{align*}
\Delta (A, \nu)     &\coloneqq \overline{\conv\{\nu'(f)/i:f\in A_i\setminus\{0\}\}} .
\end{align*}
\end{definition}
When $A$ has a finite Khovanskii basis, we may assume, without loss of generality, that every basis element is homogeneous.  In other words, $\{g_1,\dots,g_m\}$ form a Khovanskii basis and $\deg(g_i) = d_i$, that is $g_i\in A_{d_i}$.  Then, the Newton-Okounkov body is $\conv\{\nu'(g_i)/d_i\}.$

Suppose that $A = \bigoplus_{i \geq 0} A_{i}$ is a positively graded and valued algebra with valuation $\nu'$ as above with a finite Khovanskii basis $\{g_1,\dots,g_m\}$.  Let $R$ be the presentation ring such that $\pi:R\rightarrow A$ has kernel $I$, and a monomial order $<$ induced by $\nu$.

\begin{theorem}
  \label{thm:NOBodyConstruction}
  Let \(\mu\) be the valuation on \(R/I\) defined above and let \( \vec{d} = \left(\deg(g_{i}) \colon 1 \leq i \leq m \right) \) be the vector of degrees of Khovanskii basis elements.  The Newton-Okounkov body of \(R/I\) is given by
  \begin{equation}\label{eq:NewtonOkounkovLattice}
   \Delta(R/I, \mu) = \conv\left\{\|\vec{d}\|^2\mu([x_{i}])/\mu([x_i])_1\right\}.
  \end{equation}
\end{theorem}
\begin{proof}
  Fix a lattice basis \(\left(\vec{w}_{1}, \ldots, \vec{w}_{\ell}\right)\) for $K$, and set $\vec{w}_{\ell +1} = \vec{d}.$
  By construction, $\vec{w}_i$ and $\vec{w}_{\ell +1}$ are orthogonal for all $i$. Choose a set of vectors $\{\vec{w}_{\ell+2},\dots,\vec{w}_m\} \subseteq\bZ^m$ which extends $\{ \vec{w}_1,\dots,\vec{w}_{\ell +1}\}$ to a basis of the vector space $\bQ^m$ such that $\vec{w}_i$ and $\vec{w}_{\ell +1}$ are orthogonal for all $i \ne \ell +1$.  
  
  The valuation $\mu$ can be represented by $\mu:R/I\setminus\{0\}\rightarrow\bQ^{m-\ell}$ where $\mu([x^\alpha])$ are the coordinates of $\alpha$ with respect to the vectors $\{\vec{w}_{\ell+1},\dots,\vec{w}_m\}$ in the basis defined above.  This construction embeds $\torfree(\bZ^m/K)$ as a subset of $\bQ^{m-\ell}$.
  Let $L\simeq\bZ^{m-\ell}$ be the lattice generated by the images $\{\mu([x_1]),\dots,\mu([x_m])\}$, and the Newton-Okounkov body for $R/I$ is defined with respect to this lattice.  We observe that, since $\vec{w}_{\ell +1}$ is orthogonal to the other $\vec{w}_i$s, the first entry of $\mu([x^\alpha])$ is $\deg(\pi(x^\alpha))/\|\vec{d}\|^2$.
\end{proof}

Suppose $\{[x_1],\dots,[x_m]\}$ form a minimal subalgebra basis, then $\Delta(R/I, \mu)$ is $(m-\ell)$-dimensional.  When $\Delta(R/I, \mu)$ is full-dimensional, the normalized volume of the Newton-Okounkov body can be computed in terms of the standard volume on $\bQ^{m-\ell}$ (as opposed to computing the volume in terms of the integral lattice as in~\Cref{eq:NewtonOkounkovLattice}).  Since $\vec{w}_{\ell +1}$ is perpendicular to all other vectors, the multiplicative factor for the volume splits along this dimension.  In particular, the multiplicative factor for the volume is the length of a lattice generator of the first-coordinates of $\left(\mu([x_1]),\dots,\mu([x_m])\right)$ divided by the volume of a fundamental domain of $L$ in $\bQ^m$.  Since $\mu([x_i])_1=\frac{d_i}{\|\vec{d}\|}$, the length of the lattice generator of the first-coordinates is $\frac{\gcd(d_1,\dots,d_m)}{\|\vec{d}\|^2}$.

\begin{algorithm}[hbt]
  \caption{Calculating \(\vol(\Delta(R/I, \mu))\).}
  \label{alg:NOBodyVolume}
  \begin{algorithmic}[1]
    \Require a positively graded $k$-algebra and domain \(A\), valuation \(\nu\) satisfying the standing hypotheses, and finite Khovanskii basis \(\left\{g_{1}, \ldots, g_{m}\right\}\) for \((A, \nu)\).
    \Ensure the normalized volume of \(\Delta(R/I, \mu)\).
    \State Construct presentation ring \(R\) with presentation ring \(\pi: R \rightarrow A\), monomial order \(>\) induced by \(\nu\), and \(I \vcentcolon= \ker(\pi)\) (see Section~\ref{sec:KhovanskiitoSubalgebra}).
    \State Compute a Gr\"obner basis $G$ for $I$.
    \State Compute a basis $\{\vec{w}_1,\dots,\vec{w}_\ell\}$ of the lattice $K$ spanned by the toric exponents of $G$.
    \State Set \(\vec{w}_{\ell+1} = \vec{d} = \left(\deg(g_{i})\right) \in \bZ^{m}\) to be the vector of degrees of all Khovanskii basis elements.
    \State Extend $\bZ \{ \vec{w}_1, \ldots , \vec{w}_{\ell + 1} \}$ to a full-rank lattice $\bZ \{ \vec{w}_1, \ldots , \vec{w}_m \}$ such that $\vec{w}_{\ell+1} \in \{\vec{w}_{\ell+1}, \ldots , \vec{w}_m \}^{\perp }.$
    \State Set $W = (\begin{array}{c|c|c} \vec{w}_1 & \cdots & \vec{w}_m \end{array}) \in \bZ^{m\times m}.$
    \State Construct $V \in \bQ^{(m- \ell - 1) \times m}$ by selecting the last $m-\ell-1$ rows of $W^{-1}$ and scaling the $i$th column by $d_i^{-1}$ for $i=1,\ldots , m.$
    \State Construct the matrix $L'$ of minimal generators of the lattice generated by the last \(m - \ell\) rows of \(W^{-1}\).
    \State\Return $\frac{(m-\ell-1)!\gcd(d_1,\dots,d_m)\vol(\conv(V))}{\|\vec{d}\|^2\det(L')}.$
  \end{algorithmic}
\end{algorithm}

\begin{example}\label{exm:NObodies}
Consider the following example from \cite[Example 23]{BurrKhovanskiiHomotopies:2023}.  Let \(U\) denote the complex vector space of cubic polynomials in \(\mathbb{C}[x,y]\) that vanish on the points \(\left\{(4, 4), (-3, -1), (-1, -1), (3,3)\right\}\).  
We associate to $U$ the algebra \(R(U) = \bigoplus_{k \geq 0} U^{k}s^{k}\), graded by $s$-degree.
In \cite{BurrKhovanskiiHomotopies:2023}, the authors showed that \(R(U)\) has the finite Khovanskii basis \(\mathcal{B} = \left\{g_{0}s, g_{1}s, g_{2}s, g_{3}s, g_{4}s, g_{5}s, g_{6}s^{2}, g_{7}s^{3}\right\}\) under the valuation \(\nu : R(U) \setminus \{0\} \to \bQ^3 \) induced by the graded reverse lexicographic order with $x>y$, and
  \[\nu(\mathcal{B}) =
    \begin{pmatrix}
      1 & 1 & 1 & 1 & 1 & 1 & 2 & 3\\
      1 & 2 & 0 & 1 & 2 & 3 & 1 & 4\\
      1 & 0 & 3 & 2 & 1 & 0 & 3 & 1
    \end{pmatrix}.\]
  The corresponding Newton-Okounkov body \(\Delta(R(U), \nu)\) is shown in Figure \ref{exm23NOBody} and has normalized volume $5$.
  \begin{figure}[htb]
    \centering
    \begin{tikzpicture}
      \filldraw[color=blue,fill=yellow,line width=1pt] (0,3) -- (3,0) -- (2,0) -- (4/3,1/3) -- (1/2,3/2) -- cycle;
      \foreach \x in {0,...,3}
      {
        \foreach \y in {0,...,3}
        {
          \filldraw[color=red] ({\x},{\y}) circle (.06);
        }
      }
      \node[left] at (0,3) {$(0,3)$};
      \node[right] at (3,0) {$(3,0)$};
      \filldraw[color=red!60!black] (4/3,1/3) circle (.06);
      \filldraw[color=red!60!black] (1/2,3/2) circle (.06);
      \node at (1/2-.1,3/2) (A') {};
      \node at (4/3-.1,1/3) (B') {};
      \node at (-2,3/2) (A) {$\left(\frac{1}{2},\frac{3}{2}\right)$};
      \node at (-1,1/3) (B) {$\left(\frac{4}{3},\frac{1}{3}\right)$};
      \draw (A) edge[->] (A')
      (B) edge[->] (B');
    \end{tikzpicture}
    \caption{The Newton-Okounkov body \(\Delta(R(U), \nu)\)}
    \label{exm23NOBody}
  \end{figure}
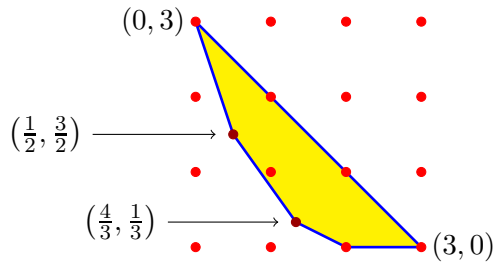
  Following Algorithm \ref{alg:NOBodyVolume} we calculate \(\vol(\Delta(R/I, \mu))\).  First let \(R = \mathbb{C}[z_{0}, z_{1}, z_{2}, z_{3}, z_{4}, z_{5}, z_{6}, z_{7}]\) be the presentation ring for \(R(V)\) and the monomial order induced on \(R\) by \(\nu\) is given by the rows of the matrix
  \[
    \begin{pmatrix}
      1 & 1 & 1 & 1 & 1 & 1 & 2 & 3\\
      2 & 2 & 3 & 3 & 3 & 3 & 4 & 5\\
      1 & 2 & 0 & 1 & 2 & 3 & 1 & 4
    \end{pmatrix}
  \]
  as weight vectors where further ties are broken with graded reverse lexicographic order.
  We have a direct sum decomposition \(\bQ^{8} = K \oplus \bQ \oplus \bQ^{2}\), with bases comprised of the columns of \(W\),
  \[
    W =
    \left(\begin{array}{ccccc|c|cc}
      1 & 2 & 3 & -3 & -4 & 1 & 0 & 0\\
      -1 & -2 & -3 & 1 & 0 & 1 & 0 & 0\\
      -1 & -1 & -1 & 0 & 1 & 1 & 0 & 0\\
      1 & 0 & 0 & 0 & 0 & 1 & 0 & 0\\
      0 & 1 & 0 & 0 & 0 & 1 & 0 & 0\\
      0 & 0 & 1 & 0 & 0 & 1 & 2 & 3\\
      0 & 0 & 0 & 1 & 0 & 2 & -1 & 0\\
      0 & 0 & 0 & 0 & 1 & 3 & 0 & -1
    \end{array}\right).
  \]
  The Newton-Okounkov body $\Delta( R/I, \mu) = \conv (V)$, where
  \[
    V =
    \begin{pmatrix}
      -\frac{11}{190} & \phantom{-}\frac{13}{38} & -\frac{91}{95} & -\frac{53}{95} & -\frac{3}{19} & \phantom{-}\frac{23}{95} & -\frac{49}{190} & \phantom{-}\frac{23}{95}\\[.1cm]
      \phantom{-}\frac{3}{190} & -\frac{7}{38} & \phantom{-}\frac{68}{95} & \phantom{-}\frac{49}{95} & \phantom{-}\frac{6}{19} & \phantom{-}\frac{11}{95} & \phantom{-}\frac{11}{95} & -\frac{62}{285}
    \end{pmatrix},
  \]
    is a polytope of Euclidean volume $\vol (V) = \frac{1}{4}$, shown in~\Cref{exm23NOBody2}.
  \begin{figure}[htb]
    \centering
    \begin{tikzpicture}[scale=2]
      \filldraw[color=blue,fill=yellow,line width=1pt] (-91/95,68/95) -- (-49/190,11/95) -- (23/95,-62/285) -- (13/38,-7/38) -- (23/95,11/95) -- cycle;
      \foreach \x in {-1,...,1}
      {
        \foreach \y in {-1,...,1}
        {
          \filldraw[color=red] ({\x},{\y}) circle (.03);
        }
      }
      \node[right] at (1,-1) {$(1,-1)$};
      \node[left] at (-1,1) {$(-1,1)$};
      \node[left] at (-91/95,68/95) {$(\frac{-91}{\phantom{-}95},\frac{68}{95})$};
      \node[right] at (13/38,-7/38) {$(\frac{13}{38}, \frac{-7}{\phantom{-}38})$};
      \filldraw[color=red!60!black] (-91/95,68/95) circle (.03);
      \filldraw[color=red!60!black] (13/38,-7/38) circle (.03);
      \filldraw[color=red!60!black] (-11/190, 3/190) circle (.03);
      \filldraw[color=red!60!black] (-53/95, 49/95) circle (.03);
      \filldraw[color=red!60!black] (-3/19, 6/19) circle (.03);
      \filldraw[color=red!60!black] (23/95, 11/95) circle (.03);
      \filldraw[color=red!60!black] (-49/190, 11/95) circle (.03);
      \filldraw[color=red!60!black] (23/95,-62/285) circle (.03);
    \end{tikzpicture}
    \caption{The Newton-Okounkov body \(\Delta(R/I, \mu)\)}
    \label{exm23NOBody2}
  \end{figure}
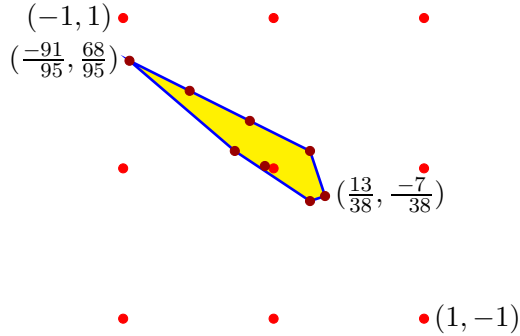
  
The lattice \( L \cong \bZ^3 \)  formed from the last three rows of $W^{-1}$ is  generated by the columns of the matrix
  \[
    L^{\prime} =
    \begin{pmatrix}
      1 & 0 & -\frac{6}{19}\\[.1cm]
      0 & 1 & -\frac{67}{190}\\[.1cm]
      0 & 0 & \phantom{-}\frac{1}{190}
    \end{pmatrix}.
  \]
Since $\vec{d}$ is the sixth column of $W$ above,~\Cref{alg:NOBodyVolume} gives us the normalized volume of $\Delta (R(U), \mu )$ as
  $$
  \frac{(m-\ell-1)!\gcd(d_1,\dots,d_8)\vol(\conv(V))}{\|\vec{d}\|^2\det(L')} = 
\frac{2!(1)\left(\frac{1}{4}\right)}{19\left(\frac{1}{190}\right)}=5,
  $$
agreeing with the normalized volume calculated in \cite[Example~23]{BurrKhovanskiiHomotopies:2023}.
\end{example}


\begin{theorem}\label{thm:affine-equivalence}
Let $\nu : A \setminus \{ 0 \} \to \bZ^{r+1}$ be the valuation on a graded algebra $A = \oplus_{i\ge 0} A_i$ induced by a valued algebra $\nu ' : A' \setminus \{ 0 \} \to \bZ^r$ satisfying the standard hypotheses with $A_i \subseteq A'$.
Assume $A$ has a finite Khovanskii basis $\{ g_1, \ldots , g_m \},$ and define $\mu:R/I\rightarrow\torfree(\bZ^m/K)$ as in~\Cref{sec:subalgebra-to-khovanskii}.
The Newton-Okounkov bodies $\Delta(A,\nu)$ and $\Delta(R/I,\mu )$ are both rational polytopes which are affinely-equivalent.
\end{theorem}
To be clear,~\Cref{thm:affine-equivalence} states that there is an affine transformation taking one Newton-Okounkov body to the other, and that this affine transformation is invertible when restricted to the $\bQ$-affine spans of the respective Newton-Okounkov bodies.
\begin{proof}
We write the linear transformation $\phi$ from Theorem \ref{thm:affinetransformation}, which takes $\mu([x_i])$ to $\nu(\pi(x_i))$, in terms of the coordinates presented in this section.  Since one coordinate of each of $\mu([x_i])$ and $\nu(g_i)$ records the degree of $[x_i]$ or $g_i$, respectively, the standard matrix for $\phi$ decomposes (after a suitable permutation of the coordinates) as
$$
\begin{pmatrix}
\ast&0^T\\
b&M
\end{pmatrix},
$$
where $\ast$ denotes the nonzero scaling factor between the two representations of the degrees.  Using this notation, and scaling by degrees, the appropriate affine transformation from $\Delta(R/I,\mu)$ to $\Delta(A,\nu)$ is $x\mapsto Mx+b.$  
\end{proof}

\bibliographystyle{plain}
\bibliography{references}

\end{document}